\documentclass[11pt]{elsarticle}

\usepackage{amsmath,amsthm,amsfonts,amssymb,graphicx}
\usepackage[hypertex]{hyperref}
\usepackage[all]{xy}
\usepackage{thmtools,thm-restate}
\usepackage[a4paper,textwidth=15cm,textheight=25cm]{geometry}
\usepackage{natbib}
\usepackage{pifont}
\newtheorem{thm}{Theorem}[section]
\newtheorem*{thmno}{Theorem}
\newtheorem{lem}[thm]{Lemma}
\newtheorem{cor}[thm]{Corollary}
\newtheorem{pro}[thm]{Proposition}

\theoremstyle{definition}
\newtheorem{defi}[thm]{Definition}

\newtheorem{exam}[thm]{Example}
\newtheorem{exams}[thm]{Examples}
\theoremstyle{remark}
\newtheorem{rem}[thm]{Remark}

\begin{document}
\begin{frontmatter}

\title{Complexity volumes of splittable groups}
\author[ms]{Mihalis Sykiotis}% \fnref{fn1}}
%\fntext[ms]{This research did not receive any specific grant from
%funding agencies in the public, commercial, or not-for-profit
%sectors.}

\ead{msykiot@math.uoa.gr}
\address{Department of Mathematics, National and Kapodistrian University of Athens, Panepistimioupolis,
GR-157 84, Athens, Greece}

\begin{abstract} Using graph of groups decompositions of finitely generated groups,
we define Euler characteristic type invariants which are non-zero in
many interesting classes of hyperbolic, limit and CSA groups,
including elementarily free groups and one-ended torsion-free
hyperbolic groups whose JSJ decomposition contains a maximal hanging
Fuchsian vertex group.
\end{abstract}


\end{frontmatter}

\section{Introduction}
Following \cite{Chi}, a \emph{generalised Euler characteristic} (or
volume) on a class of groups $\mathcal{C}$, closed under subgroups
of finite index, is a function $\chi:\mathcal{C}\rightarrow R$,
where $R$ is a commutative ring, satisfying the following condition:
If $G\in \mathcal{C}$ and $H$ is a subgroup of finite index in $G$,
then $\chi(H)=[G:H]\chi(G)$.

The existence of such an invariant, non-vanishing on a group $G$ has
two immediate, but important, consequences. First, isomorphic
subgroups of finite index of $G$ have the same index. Second, every
monomorphism from G to G with image of finite index is an
automorphism.

Euler characteristics have been studied by Bass, Brown, Chiswell,
Serre, Hattori and Stallings (see \cite{Chi} and the references
contained therein) on classes of groups satisfying certain
homological finiteness properties.

The idea of defining generalised Euler characteristics on groups
using a limiting process and subgroups of finite index goes back to
Wall \cite[Problem E10, p. 385]{ScW}. More precisely, given a real
valued function $f$ on a class of groups such that $f(H)\leq [G:H]
f(G)$, for any subgroup $H$ of finite index in $G$, Wall defined
$$\tilde{f}(G)=\textrm{inf}\bigg\{\frac{f(H)}{[G:H]}\,:\, H \textrm{ subgroup of finite index in } G \bigg\}$$
and remarked that $\tilde{f}$ is a generalised Euler characteristic.
He also gave two examples of functions that satisfy the inequality
above: the minimum number of generators for a group $G$ and the
minimum (over all presentations) of total of lengths of relators.
Examples of invariants constructed using limits and subgroups of
finite index can be found in \cite{Rez}, \cite{Sy}, \cite{Lac},
\cite{BK}, \cite{Ts}.

In this paper, we follow the approach initiated in \cite{Sy} to
define invariants using graph of groups decompositions of groups.
Let $G$ be a group acting on a tree $X$ (without inversions). By
Bass-Serre theory, for which we refer to \cite{DicksDun,Serre}, this
is equivalent to saying that $G$ is the fundamental group of the
corresponding graph of groups $(\mathcal{G},X/G)$. We also say that
the pair $(\mathcal{G},X/G)$ is a splitting of $G$.

A vertex $v$ of $X/G$ is called \emph{degenerate} if there is an
edge $e$ incident to $v$ with $G_{v}=G_{e}$. We denote by
$V_{ndeg}(X/G)$ the set of non-degenerate vertices of $X/G$, and by
$r(X/G)$ the rank of the free group $\pi_{1}(X/G)$. The
\emph{complexity} $C_{X}(G)$ (or simply $C(G)$ when $X$ is clear
from context) of $G$ with respect to the above action is defined to
be the sum $C_{X}(G)=r(X/G)+|V_{ndeg}(X/G)|$, if $G$ contains
hyperbolic elements, and $1$ otherwise.

In the case where $G$ is finitely generated the above sum is finite.
To see this, let $X_{G}$ be the minimal $G$-invariant subtree of
$X$, which is the union of the axes of all hyperbolic elements of
$G$. The fact that $G$ is finitely generated implies that the
quotient graph $X_{G}/G$ is finite.  On the other hand, by
\cite[Prop. 2.2]{Sy}, the subtree $X_{G}$ is a ``core" for the
action of $G$ on $X$ in the sense that
$r(X/G)+|V_{ndeg}(X/G)|=r(X_{G}/G)+|V_{ndeg}(X_{G}/G)|$. It follows
that $C_{X}(G)$ is finite as claimed.

The above definition of complexity differs slightly from that in
\cite{Sy} in what the complexity of an elliptic action (i.e. each
group element is elliptic) is now defined to be $1$ instead of $0$.
It seems that $1$ is more appropriate for our purposes. Moreover, in
the case of an elliptic action of a finitely generated group $G$,
any minimal $G$-invariant subtree consists of a single vertex.

Suppose now that vertex groups are contained in disjoint classes
$\mathcal{C}_{i}$ of groups each of which is endowed with a volume
$\varphi_{i}$. To simplify notation we will denote each
$\varphi_{i}$ with $\varphi$.

Suppose first that the group $G$ acting on $X$ contains hyperbolic
elements. As usual, we denote by $G_{v}$ the stabilizer of the
vertex $v$. In this case, we define the \emph{weighted complexity}
$C^{\varphi}(G)$ of $G$ with respect to the given action on $X$ by
the formula
\[C^{\varphi}(G)=r(X/G)+\sum[1+\varphi(G_{v})],\]
where the sum is over all non-degenerate vertices of $X/G$, provided
that the rank of $X/G$ is finite and the sum converges (this is the
case if $G$ is finitely generated). As before if each element of $G$
is elliptic, then the complexity is defined to be $1$. In the case
where $\varphi$ is the trivial volume (i.e. $\varphi$ is identically
zero) the above formula defines the usual complexity $C_{X}(G)$ of
$G$.

Let $\mathcal{S}$ be a class of small groups (in the sense of
Bestvina and Feighn \cite{BF}) closed under subgroups and let
$\mathcal{V}$ be any class of groups closed under subgroups of
finite index. The \textsl{volume} $V(G)$ of a finitely presented
group (or almost finitely presented) $G$ with respect to
$\mathcal{S}$ and $\mathcal{V}$ is defined to be the upper limit
$$\underset {H\in N_{G}}{\varlimsup}\frac{C_{max}(H)}{
[G:H]}\in [0,\infty],$$ where $N_{G}$ is the set of all normal
subgroups of finite index in $G$ and $C_{max}(H)$ is the maximal
complexity of $H$ over all reduced splittings of $H$ with vertex
groups in $\mathcal{V}$ and edge groups in $\mathcal{S}$.
Bestvina-Feighn's accessibility theorem \cite{BF} ensures that
$V(G)$ is finite (see Proposition \ref{finite vol}).

Using splittings over finite subgroups, we obtain the volume
$V_{fin}$ which is strictly positive on any finitely presented (or
accessible) group with infinitely many ends and provides us with a
basic tool for constructing volumes on one-ended groups.

In Section \ref{Acyl}, we consider acylindrical splittings of
finitely generated, one-ended groups over finitely generated, free
abelian groups of rank at most $n$. We use any ``suitable" volume
$\varphi$ on the vertex groups (see Def. \ref{AcyclVol}) and work
with the corresponding weighted complexity. In this way, we define
volumes $V^{\varphi}_{n}$ on finitely generated, one-ended groups,
using reduced acylindrical splittings as above, of maximal weighted
complexity. In this case, finiteness of the volume follows from
acylindrical accessibility \cite{Sel1,We1} and an analogue of
Grushko's theorem for acylindrical splittings \cite{We2}.

Our main results are the following.
\begin{restatable*}{thm}{firstFinVolOne}\label{FinVolOne}
Suppose that $G$ is a finitely generated one-ended group which
admits an acylindrical splitting $(\mathcal{G},Y)$ over free abelian
groups of bounded rank, say by a positive integer $n$. If there is a
vertex group $G_{v}$ with $\varphi(G_{v})>0$ (where $\varphi$ is as
above), then $V^{\varphi}_{n}(G)>0$.
\end{restatable*}

\begin{restatable*}{thm}{firstImagStab}\label{ImagStab}
Let G be a one-ended finitely generated group such that
$V^{\varphi}_{n}(G)>0$. If $f$ is an endomorphism of $G$ whose image
is of finite index, then the decreasing sequence of images
$f^{k}(G)$, $k\in \mathbb{N}$, is eventually constant. Moreover, if
$G$ is torsion-free and every subgroup of finite index is Hopfian
(e.g. $G$ residually finite or hyperbolic), then $f$ is an
automorphism.
\end{restatable*}

We also prove an analogous result for groups with infinitely many
ends, which generalizes a result of Hirshon \cite{Hi}.
\begin{restatable*}{pro}{firstStabImInf}  \label{StabImInf}
Let $G$ be a finitely generated group with
infinitely many ends. Suppose that $G$ splits as an amalgam or a HNN
extension over a malnormal subgroup and that each subgroup of finite
index of $G$ is Hopfian. Then each endomorphism of $G$ with image of
finite index is an automorphism.
\end{restatable*}

Groups with the property that every endomorphism with image of
finite index is an automorphism are called \textsl{cofinitely
Hopfian} in \cite{BGHM}.

In Sections \ref{SecHyp}, \ref{SecLim} and \ref{SecCSA}, we use the
volume $V_{fin}$ on vertex groups with infinitely many ends, and
show how the above results can be applied in torsion-free,
one-ended, hyperbolic groups, limit groups and, more generally, CSA
groups (a group $G$ is CSA if maximal abelian subgroups are
malnormal). In each of these cases, the existence of ``surface type"
vertex groups in such a splitting implies that its fundamental group
has positive volume.

More precisely, in the case of torsion-free, one-ended hyperbolic
groups, we consider cyclic splittings (i.e. splittings where each
edge group is infinite cyclic) and denote the corresponding volume
by $V_{1}^{hyp}$.
\begin{restatable*}{pro}{firsthyppos} \label{hyppos}
Let $G$ be a torsion-free, one-ended hyperbolic group which admits a
cyclic splitting with a free non-abelian vertex group. Then
$V_{1}^{hyp}(G)>0$.
\end{restatable*}

As a special case we have:
\begin{restatable*}{cor}{firstJSJpos}\label{JSJpos}
Let $G$ be a torsion-free, one-ended hyperbolic group. If  the JSJ
decomposition of $G$ contains a (MHF) subgroup, then
$V_{1}^{hyp}(G)>0$. In particular, if the outer automorphism group
$\textrm{Out}(G)$ of $G$ is not virtually finitely generated free
abelian, then $V_{1}^{hyp}(G)>0$.
\end{restatable*}

It should be noted that there are hyperbolic groups as above with
strictly positive volume and trivial JSJ decomposition. For example,
if $G$ is the fundamental group of a closed orientable surface
$S_{g}$ of genus $g\geq 2$, then $V_{1}^{hyp}(G)\geq 2(g-1)$ (see
Example \ref{surfacetype}).

In the case of one-ended limit groups, we consider cyclic
$2$-acylindrical splittings, which arise naturally from the
characterization of limit groups as finitely generated subgroups of
$\omega$-residually free tower groups. Let $V_{1}^{lim}$ denote the
corresponding volume.

\begin{restatable*}{pro}{firstlimpos}\label{limpos}
Let $H$ be a one-ended limit group of height $h\geq 1$ and
$X_{h}$ an $\omega$-rft space such that the fundamental group of
$X_{h}$ contains a copy of $H$.  If the final block (in the
construction of $X_{h}$) is quadratic, then $V_{1}^{lim}(H)>0$.
\end{restatable*}

In particular, if $G$ is a one-ended elementarily free group, then
$V_{1}^{lim}(G)>0$.

In order to generalize the above results to the case of
torsion-free, one-ended CSA groups, we use the work of Guirardel and
Levitt  in \cite{GL1}, where a method is given for constructing an
acylindrical $G$-tree $T_{c}$ (the tree of cylinders) from any
$G$-tree $T$, for any finitely generated group $G$. If $G$ is a
torsion-free, one-ended CSA group and the $G$-tree $T$ has finitely
generated abelian edge stabilizers, then the edge stabilizers of
$T_{c}$ are non-trivial abelian and $T_{c}$ is $2$-acylindrical.
Moreover, $T$ and $T_{c}$ have the same non-abelian stabilizers
which are divided into two types: rigid and flexible. Non-abelian
flexible vertex stabilizers are fundamental groups of compact
surfaces with boundary. Let $V_{n}^{CSA}$ denote the volume defined
using acylindrical splittings over abelian subgroups of bounded rank
$n$, of one-ended, torsion-free CSA groups.
\begin{restatable*}{pro}{firstCSA}\label{CSA}
Let $G$ be a finitely generated, torsion-free, one-ended CSA group
such that abelian subgroups of $G$ are finitely generated of bounded
rank. Suppose that $G$ admits a splitting over abelian subgroups
such that the associated tree of cylinders $T_{c}$ has a non-abelian
flexible vertex stabilizer $G_{v}$. Then there exists $n$ such that
$V_{n}^{CSA}(G)>0$.
\end{restatable*}

Beside torsion-free hyperbolic groups and limit groups, other
examples of Hopfian, CSA groups are $\Gamma$-limit groups, where
$\Gamma$ is a torsion-free group relatively hyperbolic to a finite
family of finitely generated abelian subgroups \cite{Gro1,Gro2}.

\begin{restatable*}{pro}{firstCSAlimit}\label{CSAlimit}
Let $\Gamma$ be a torsion-free group which is hyperbolic relative to
a finite collection of free abelian subgroups and let $G$ be a
one-ended $\Gamma$-limit group. Suppose that $G$ admits a splitting
over abelian subgroups which has a non-abelian vertex group
isomorphic to the fundamental group of a compact surface with
boundary. Then there exists a positive integer $n=n(G)$ such that
$V_{n}^{CSA}(G)> 0$.
\end{restatable*}

\section{Splittings and Volumes}

A finitely generated group $G$ is \emph{small} if it does not admit
a minimal action on a tree for which there are two hyperbolic
elements whose axes intersect in a compact set. Each finitely
generated group not containing the free group of rank 2 is small. A
splitting of $G$ as fundamental group of a graph of groups is
\emph{reduced} if for each edge $e$ not a loop, the edge group
$G_{e}$ is contained properly in the vertex groups of its endpoints.
This means that for each degenerated vertex the corresponding
isomorphism is induced by an edge which is a loop. A $G$-tree is
\emph{reduced} if the corresponding graph of groups is reduced. In
\cite{BF} the term means something weaker: the vertex group of each
vertex of valence two properly contains the corresponding edge
groups.

Let $\mathcal{S}$ be a class of small groups closed under subgroups
and $\mathcal{V}$ a class of groups closed under subgroups of finite
index. For each finitely presented group $\Gamma$, let
$C_{max}(\Gamma)$ denote the maximal complexity of $\Gamma$ over the
complexities of all reduced splittings of $\Gamma$ over
$\mathcal{S}$ with vertex groups in $\mathcal{V}$. Since the
complexity of a splitting can be computed using the corresponding
minimal invariant subtree, we may suppose that all splittings are
minimal.

The main result in \cite{BF} insures that $C_{max}(\Gamma)<\infty$.
More specifically, Bestvina and Feighn proved that for each finitely
presented group, there is an integer $\gamma(\Gamma)$ such that for
each reduced $\Gamma$-tree $X$ with small edge stabilizers, the
number of vertices in the quotient graph $X/\Gamma$ is bounded above
by $\gamma(\Gamma)$. Since the rank of $X/\Gamma$ is bounded above
by the minimal number of generators of $\Gamma$, it follows that
$C_{max}(\Gamma)<\infty$.

We denote by $N_{G}$ the set of all normal subgroups of finite index
in $G$, partially ordered by inverse inclusion.

\begin{defi}\label{defvol}  Let $G$ be a finitely presented group.
The \textsl{volume} $V(G)$ of $G$ with respect to $\mathcal{S}$ and
$\mathcal{V}$ is the upper limit $\underset {H\in
N_{G}}{\varlimsup}\displaystyle\frac{ C_{max}(H)}{[G:H]}\in
[0,\infty]$, i.e. the supremum of the set $L_{N_{G}}=\Big\{\ell\in
\mathbb{R}: \Big(\frac{C_{max}(K)}{[G:K]}\Big)_{K\in \mathcal{K}}
\rightarrow \ell \;\;\textrm{for some cofinal subset }\;\; K
\textrm{ of } N_{G} \Big\}$.
\end{defi}
It is not difficult to show that $L_{N_{G}}\neq \emptyset$. We will
see in Proposition \ref{finite vol} that $V(G)$ is always finite.
\begin{rem}
If $H$ is a finite index subgroup of $G$, then $H$ is finitely
presented and therefore $C_{max}(H)<\infty$.
\end{rem}

\begin{rem}
We could also use any cofinal subset of the collection $\Lambda_{G}$
of all subgroups of finite index in $G$ (partially ordered by
inverse inclusion) instead of $N_{G}$ (or other limit point of the
net). But, calculations are often simplified by working with normal
subgroups.
\end{rem}

\begin{rem}
In the case where $G$ contains finitely many subgroups of finite
index, then $N_{G}$ has a minimum element, say $H_{0}$, and
$V(G)=\displaystyle\frac{ C_{max}(H_{0})}{[G:H_{0}]}$. In
particular, if $G$ is finite, then
$V(G)=\displaystyle\frac{1}{|G|}$.
\end{rem}
\begin{rem}\label{inflimindex}
If $G$ contains infinitely many subgroups of finite index and
$\mathcal{K}$ is a cofinal subset of $N_{G}$ with
$\Big(\frac{C_{max}(K)}{[G:K]}\Big)_{K\in \mathcal{K}}\rightarrow
\ell$, then there exists a sequence $K_{n}$ in $\mathcal{K}$ such
that $\Big(\frac{C_{max}(K_{n})}{[G:K_{n}]}\Big)_{n\in
\mathbb{N}}\rightarrow \ell$ and $[G:K_{n}]\rightarrow \infty$.
Indeed, we consider the sequence $(H_{n})$, where $H_{n}$ is the
intersection of all subgroups in $G$ of index at most $n$. Since $G$
is finitely generated and contains infinitely many subgroups of
finite index, it follows that $[G:H_{n}]\rightarrow \infty$. The
cofinality of $\mathcal{K}$ implies that we can choose for each $n$
a subgroup $K_{n}\in \mathcal{K}$ such that $K_{n}\subseteq H_{n}$
and thus $[G:K_{n}]\rightarrow \infty$. On the other, this sequence
is a cofinal subset of $\mathcal{K}$ and therefore
$\Big(\frac{C_{max}(K_{n})}{[G:K_{n}]}\Big)_{n\in
\mathbb{N}}\rightarrow \ell$.
\end{rem}

\begin{rem}\label{Weigfinite}
Finally, if the vertex groups are partitioned in classes (closed
under subgroups of finite index) in each of which is already defined
a volume, then we can use in the above definition the weighted
complexity.  In this case, in order to insure that the corresponding
volume $V^{\varphi}(G)$ is finite (with respect to $\mathcal{S}$ and
$\mathcal{V}$), it suffices to impose the following additional
hypothesis:

\begin{enumerate}
  \item[(H)] \label{H} For each subgroup $H$ of finite index in $G$ and each
vertex group $H_{v}$ in any splitting of $H$, with edge groups in
$\mathcal{S}$ and vertex groups in $\mathcal{V}$, which attains
$C_{max}^{\varphi}(H)$, there are constants $A_{1}$, $A_{2}$,
$A_{3}$ and $A_{4}$, independent of $H$, such that $$\sum_{v\in
V(Y)}r(H_{v})\leq A_{1}\cdot r(H)+A_{2} \,\textrm{  and  }\,
\varphi(H_{v})\leq A_{3}\cdot r(H_{v})+A_{4}.$$
\end{enumerate}

We will see in section \ref{Acyl}, that hypothesis (H) is satisfied
in many interesting cases.
\end{rem}
For a group $\Gamma$, we denote by $r(\Gamma)$ the minimum number of
generators of $\Gamma$.

\begin{pro}\label{finite vol} Under the above hypotheses, $V^{\varphi}(G)<\infty$.
In particular $V(G)<\infty$.
\end{pro}
\begin{proof}
Let $H$ be a subgroup of finite index in $G$, $(\mathcal{H},Y)$ a
reduced splitting of $H$ over small groups of maximal complexity and
$X$ the corresponding universal tree. It is proved in \cite{BF} that
the number of vertices of $Y$ is bounded above by $94 \delta(H)+233
\beta_{1}(H)+6\dim H^{1}(H;\mathbb{Z}_{2})-139$, where
$\beta_{1}(H)$ is the (torsion-free) rank of the abelianization
$H_{ab}$ of $H$ and $\delta(H)$ is a constant (see
\cite{Dun,DicksDun}) defined as follows. Let $\Gamma$ be an almost
finitely presented group. For each 2-dimensional complex $K$ such
that $H^{1}(K;\mathbb{Z}_{2})=0$, on which $\Gamma$ acts freely with
finite quotient $L$, we define $\delta_{\Gamma}(L)=2 \,\dim
H^{1}(L;\mathbb{Z}_{2})+a_{0}+a_{2}$, where $a_{i}$ is the number of
$i$-simplices of $L$. The minimum of all possible
$\delta_{\Gamma}(L)$, is denoted by $\delta(\Gamma)$.

We first note that $\beta_{1}(H)\leq r(H)\leq [G:H](r(G)-1)+1$.

To bound the number $\dim H^{1}(H;\mathbb{Z}_{2})$, we use the
isomorphisms \[H^{1}(H;\mathbb{Z}_{2})\cong
\textrm{Hom}\big(H_{1}(H),\mathbb{Z}_{2}\big)=\textrm{Hom}(H_{ab},\mathbb{Z}_{2})\cong
\bigoplus_{i=1}^{k}\textrm{Hom}(C_{i},\mathbb{Z}_{2}),\] where
$C_{i}$ are cyclic groups whose direct sum is $H_{ab}$, i.e.
$H_{ab}=C_{1}\oplus \cdots \oplus C_{k}$. Since
$\textrm{Hom}(\mathbb{Z},\mathbb{Z}_{2})=\mathbb{Z}_{2}$ and
$\textrm{Hom}(\mathbb{Z}_{n},\mathbb{Z}_{2})=\mathbb{Z}_{2}[n]=\{g\in
\mathbb{Z}_{2}\,:\,ng=0 \}$, it follows that
$H^{1}(H;\mathbb{Z}_{2})$ is the direct sum of at most $k$ copies of
$\mathbb{Z}_{2}$ and thus $\dim H^{1}(H;\mathbb{Z}_{2})\leq k$. The
well-known proof of the existence of the above decomposition of
$H_{ab}$ (with induction on its rank) as direct sums of cyclic
groups, actually shows that the number of summands is equal to its
rank. Hence $\dim H^{1}(H;\mathbb{Z}_{2})\leq k=r(H_{ab})\leq
r(H)\leq [G:H](r(G)-1)+1$.

Fix a finite 2-dimensional complex $L$ with fundamental group $G$.
Let $a_{i}$ denote the number of $i$-simplices of $L$. The group $G$
acts freely on the universal cover $K$ of $L$ with finite quotient.
Moreover, $H^{1}(K;\mathbb{Z}_{2})=0$. Since $H$ is of finite index
in $G$, it acts freely on $K$ with finite quotient $L_{1}$. If
$b_{i}$ is the number of $i$-simplices of $L_{1}$, then
$b_{i}=[G:H]a_{i}$. Therefore,
\begin{align*}\delta(H)& \leq
\delta_{H}(L_{1})=2 \dim H^{1}(L_{1};\mathbb{Z}_{2})+b_{0}+b_{2}= 2
\dim
\textrm{Hom}(H_{1}(L_{1});\mathbb{Z}_{2})+b_{0}+b_{2}\\
 &= 2 \dim
\textrm{Hom}(H_{ab};\mathbb{Z}_{2})+b_{0}+b_{2}\leq
2[G:H](r(G)-1)+1+[G:H]a_{0}+[G:H]a_{2}.\end{align*}
 From the analysis above, we see that there are constants $A$, $B$ and $C$,
such that $$|V(Y)|\leq A \cdot [G:H](r(G)-1)+B+C\cdot
[G:H](a_{0}+a_{2}).$$
 Now, $$\sum_{v\in V(Y)}(1+\varphi(H_{v})\leq \sum_{v\in V(Y)}(1+A_{3}\cdot r(H_{v})+A_{4})
 \leq A_{3}\big(A_{1}\cdot r(H)+A_{2}\big)+|V(Y)|\cdot (A_{4}+1).$$
Since $r(Y)\leq r(H)$, we conclude that there are constants $A'$,
$B'$ and $C'$, such that
$$C_{max}^{\varphi}(H)\leq r(Y)+\sum_{v\in V(Y)}(1+\varphi(H_{v})\leq A' \cdot [G:H](r(G)-1)+B'+C'\cdot
[G:H](a_{0}+a_{2}).$$ It follows that the net
$\displaystyle\bigg(\frac{C_{max}^{\varphi}(H)}{[G:H]}\bigg)_{H\in
N_{G}}$ is bounded above by $A' \cdot (r(G)-1)+\displaystyle\frac{
B'}{ [G:H]}+C'\cdot (a_{0}+a_{2})$ and hence the volume
$V^{\varphi}(G)$ of $G$ is finite.
\end{proof}

The volume of a finitely presented group (or even of an almost
finitely presented group) with respect to a family of splittings as
above, is a generalized Euler characteristic in the following sense:

\begin{pro}\label{multindex}
Let $H$ be a subgroup of finite index in a finitely presented group
$G$. Then $V^{\varphi}(H)=[G:H]V^{\varphi}(G)$. In particular, if
$V^{\varphi}(G)\neq 0$, then isomorphic subgroups of finite index of
$G$ have the same index and every monomorphism from $G$ to $G$ with
image of finite index is an automorphism.
\end{pro}
\begin{proof}
\[V^{\varphi}(H)=\underset {K\in
N_{H}}{\varlimsup}\frac{C_{max}(K)}{ [H:K]}= [G:H]\underset {K\in
N_{H}}{\varlimsup}\frac{ C_{max}(K)}{[G:K]}=[G:H]V^{\varphi}(G),\]
where the first equality follows from the definition of
$V^{\varphi}$ and the last one from the fact that for each subgroup
$K$ of finite index in $G$, there is a normal subgroup $M$ of $G$ of
finite index such that $M\subseteq K\cap H$.
\end{proof}

We end this section by adapting the definition of complexity volume
for finitely presented groups with infinitely many ends from
\cite{Sy} to the previous setting (which is essentially the same).
It is the volume that is used mainly to define the weighted
complexity of splittings of one-ended groups. We shall also present
some calculations which lead to explicit formulae for the volume of
finitely presented virtually torsion-free groups. These formulae
show that our approach gives a well-behaved generalised Euler
characteristic.

We consider reduced splittings of a finitely presented group over
finite subgroups (without restrictions on vertex groups) and denote
by $V_{fin}(G)$ the corresponding volume. It is worth noting that
the maximal complexity among such reduced splittings of a finitely
generated, torsion-free group $G$ is equal to the number of factors
in the Grushko decomposition of $G$ (i.e. each factor is freely
indecomposable).

\begin{rem} It is not difficult to see that $V_{fin}(G)$ can be also
defined using splittings over finite subgroups and vertex groups
with at most one end, and that it is a finite number if one assumes
only that $G$ is accessible (this follows from \cite[Lemma
7.6]{ScW}). In particular, by Linnell's accessibility theorem
$V_{fin}(G)$ is finite for each finitely generated virtually torsion
free group $G$. The case of free products, i.e. splttings over the
trivial group, is studied in more detail in \cite{Ts}.
\end{rem}

\begin{exams}
(i) Let $G=G_{1}\ast \cdots \ast G_{n}$ be a free product of
torsion-free, freely indecomposable groups, satisfying the condition
in Remark \ref{inflimindex}. The assumption on the free factors
implies that the splitting of any finite index subgroup $H$ of $G$
inherited from the given one of $G$, is of maximal complexity. By
\cite[Prop. 3.2]{Sy}, $C_{max}(H)-1=[G:H](n-1)$ from which it
follows that $V_{fin}(G)=n-1$. In particular, if $F_{n}$ is the free
group of rank $n$, then $V_{fin}(F_{n})=n-1$.\\
(ii) Let $G$ be a finitely presented group with infinitely many
ends, let $(\mathcal{G},Y)$ be a nontrivial finite graph of groups
decomposition of $G$ with finite edge groups such that each vertex
group has at most one end, and let $X$ be the corresponding
universal tree. Suppose that $G$ contains a normal, torsion-free
subgroup $H$ of finite index, satisfying the condition in Remark
\ref{inflimindex} and that $v_{1}\ldots, v_{m}$ are the vertices of
$Y$ with finite vertex group. The subgroup $H$ acts edge freely on
$X$ and each non-trivial vertex stabilizer is infinite and freely
indecomposable being of finite index in the corresponding
$G$-stabilizer. Therefore $V_{fin}(H)=n-1$, where $n$ is the number
of free factors of $H$ in the free product decomposition associated
to the above action, i.e.
\[\begin{array}{ccl}
n-1& = &
C_{X}(H)-1=r(X/H)+|V_{ndeg}(X/H)|-1\\
  & = & |E_{+}(X/H)|-|V(X/H)|+1+|V_{ndeg}(X/H)|-1\\
 &=& |E_{+}(X/H)|-|V_{deg}(X/H)|=\sum_{e\in E_{+}(Y)}|H\backslash
 G/G_{e}|-\sum_{i=1}^{m}|H\backslash G/G_{v_{i}}|\\
 & \stackrel{H\lhd G}{=} & \sum_{e\in
 E_{+}(Y)}\frac{[G:H]}{|G_{e}|}-\sum_{i=1}^{m}\frac{[G:H]}{|G_{v_{i}}|}\\
\end{array}\]
It follows that $$V_{fin}(G)=\frac{V_{fin}(H)}{[G:H]}=\sum_{e\in
 E_{+}(Y)}\frac{1}{|G_{e}|}-\sum_{i=1}^{m}\frac{1}{|G_{v_{i}}|}$$ (compare with \cite[Cor. 3.4]{Ts}
and \cite[Rem. IV.1.11(vii)]{DicksDun}).

In particular, if $\Gamma=\Gamma_{1}\ast_{A}\Gamma_{2}$, where
$\Gamma_{1}$ and $\Gamma_{2}$ are finite groups, then
$V_{fin}(\Gamma)=\displaystyle\frac{1}{|A|}-\frac{1}{|\Gamma_{1}|}-\frac{1}{|\Gamma_{2}|}$.
Also, if $G=\underbrace{G_{1}\ast \cdots \ast
G_{n}}_{\textrm{torsion-free factors}}\ast \underbrace{G_{n+1}\cdots
\ast G_{n+m}}_{\textrm{finite factors}}$ is a free product of freely
indecomposable groups, satisfying the condition in Remark
\ref{inflimindex}, then
$$V_{fin}(G)=n+m-1-\sum_{i=n+1}^{n+m}\frac{1}{|G_{i}|}.$$
(iii) Let $G=A\ast_{H} B$ or $A\ast_{H}$, where $A$, $B$ are
finitely presented virtually torsion-free groups satisfying the
condition in Remark \ref{inflimindex} and $H$ is a finite group.
Following the proof of \cite[Theorem 7.3]{ScW}, it is easy to show
that $G$ is virtually torsion-free and therefore the above formula
can be used. If $(\mathcal{G}_{A},Y_{A})$ and
$(\mathcal{G}_{B},Y_{B})$ are finite graph of group decompositions
of $A$ and $B$, respectively, then we can construct a graph of
groups decomposition of $G$ from these by attaching one edge $e$,
with associated edge group $G_{e}$ such that $|G_{e}|=|H|$. Now, the
formula in the above example implies that
$V_{fin}(A\ast_{H}B)=V_{fin}(A)+V_{fin}(B)+\frac{1}{|H|}$ and
$V_{fin}(A\ast_{H})=V_{fin}(A)+\frac{1}{|H|}$. Both formulas (up to
a sign) have been obtained recently in \cite{Ts}, in a more general
context, using covering space arguments.
\end{exams}

Given a positive integer $m$, we denote by $\mathcal{S}_{m}$ the
class of finite groups of order at most $m$ and by $V_{fin,m}(G)$
the volume of a group $G$ with respect to $\mathcal{S}_{m}$ (i.e.
using splittings of $G$ over $\mathcal{S}_{m}$). It is immediate
from the definition that $V_{fin}(G)\geq \cdots \geq
V_{fin,m+1}(G)\geq V_{fin,m}(G)$.

For later use we need the following lemma.
\begin{lem}\label{VfinH1}
Let $G$ be a finitely generated group. Then $V_{fin,m}(G)\leq m\cdot
r(G)+1$. If, moreover, there is a bound $C$ on the orders of the
finite subgroups of $G$, then $V_{fin}(G)\leq C\cdot r(G)+1$.
\end{lem}
\begin{proof}
Let $H$ be a subgroup of $G$ of finite index and $(\mathcal{H}, Y)$
a reduced splitting of $H$ over $S_{m}$ of maximal complexity. Then,
by \cite[Theorem 2.3]{Dun1} (see also \cite{Li}), $\sum_{e\in
E_{+}(Y)}\frac{1}{|H_{e}|}\leq r(H)$ and thus $|E_{+}(Y)|\leq m\cdot
r(H)$. It follows that $C_{max}(H)\leq r(Y)+|V(Y)|=|E_{+}(Y)|+1\leq
m\cdot r(H)+1\leq [G:H]m\cdot (r(G)-1)+m+1$. By dividing by $[G:H]$,
we see that
$$\frac{C_{max}(H)}{[G:H]}\leq m\cdot
(r(G)-1)+\frac{m+1}{[G:H]}\leq m\cdot r(G)+1.$$ Since the net
$\Big(\scalebox{1.1}{$\frac{C_{max}(H)}{[G:H]}$}\Big)_{H\in N_{G}}$
is bounded above by $m\cdot r(G)+1$, the first claim follows. The
second claim is proved in the same way.
\end{proof}

\section{The case of one-ended groups-Acylindrical
splittings}\label{Acyl}

Following Sela \cite{Sel1}, an action of a group $G$ on a tree $T$
is said to be $k$-\textsl{acylindrical} if the stabilizer of any
path of length greater than $k+1$ is trivial. A graph of groups
decomposition of $G$ is $k$-\textsl{acylindrical} if the action of
$G$ on the corresponding universal tree is $k$-acylindrical. For
example, canonical JSJ decompositions of hyperbolic groups and
cyclic decompositions of limit groups are acylindrical (in fact,
2-acylindrical).

In this section, we consider reduced $k$-acylindrical splittings of
one-ended finitely generated groups over finitely generated free
abelian subgroups of rank at most $n$, and use weighted complexity.

\begin{defi}\label{AcyclVol} Let $\mathcal{C}$ be a class of groups (closed under subgroups
of finite index) and let $\varphi_{0}$ be a generalized Euler
characteristic on $\mathcal{C}$ satisfying $\varphi_{0}(\Gamma)\leq
A_{3}\cdot r(\Gamma)+A_{4}$ for some constants $A_{3}$ and $A_{4}$.
Let $G$ be a finitely generated one-ended group and
$(\mathcal{G},Y)$ a reduced $k$-acylindrical splitting of $G$ such
that each edge group $G_{e}$ is free abelian with $r(G_{e})\leq n$,
for a given positive integer $n$.  If $G_{v}$ is a vertex group of
$(\mathcal{G},Y)$, then we define
$\varphi(G_{v})=\varphi_{0}(G_{v})$ if $G_{v}$ is in $\mathcal{C}$
and $\varphi(G_{v})=0$, otherwise.

Let $V^{\varphi}_{n}(G)= \underset {H\in
N_{G}}{\varlimsup}\displaystyle\frac{C_{max}^{\varphi}(H)}{
[G:H]}\in [0,\infty]$ be the associated volume.
\end{defi}

\begin{rem} The inequality in the above definition is satisfied by
$V_{fin}$ when restricted to finitely generated groups with a
uniform bound on the order of finite subgroups (see Lemma
\ref{VfinH1}) as well as by the deficiency and rank volumes defined
in \cite{Rez} (with $A_{3}=1$ and $A_{4}=-1$).
\end{rem}

Suppose for a moment that $G$ is finitely presented. Then, in order
to have that $V_{n}^{\varphi}(G)<\infty$, it suffices, by
Proposition \ref{finite vol}, to show that the sum of ranks of
non-degenerate vertex groups in any $k$-acylindrical splitting is
bounded above by a linear expression of the rank of the fundamental
group. For this we need the following Theorem from \cite{We2}.

\begin{thm}[\cite{We2}, Theorem 0.1]\label{ThmGru}
Let $(\mathcal{G},Y)$  be a $k$-acylindrical, minimal and without
trivial edge groups splitting of a finitely generated group $G$.
Then
$$ r(G)\geq \frac{1}{2k+1}\Bigg( \sum_{v\in V(Y)}r(G_{v})-\sum_{e\in E_{+}(Y)}r(G_{e})+
2|E_{+}(Y)|+r(Y)+1+3k-\left \lfloor{k/2}\right \rfloor  \Bigg),$$
where $\left \lfloor{k/2}\right \rfloor$ is the largest integer less
than or equal to $k/2$.
\end{thm}

Moreover, using acylindrical splittings, we can obtain finiteness of
$V_{n}^{\varphi}(G)$ for a finitely generated group $G$, bypassing
the proof of Proposition \ref{finite vol}, because of the following
accessibility result.
\begin{thm}[\cite{Sel1,We1}]\label{AcylAccess}
Let $G$ be a non-cyclic, freely indecomposable finitely generated
group which admits a minimal and $k$-acylindrical action on a tree
$T$. Then the quotient graph $T/G$ has at most $2k(r(G)-1)+1$
vertices.
\end{thm}

Now, combining acylindrical accessibility and Theorem \ref{ThmGru},
we are in position to prove the finiteness of $V^{\varphi}_{n}(G)$.

\begin{pro}
Let $G$ be a one-ended finitely generated group. Then
$V^{\varphi}_{n}(G)<\infty$.
\end{pro}
\begin{proof}
Let $H$ be a finite index subgroup of $G$ and $(\mathcal{H},Y)$ a
reduced $k$-acylindrical splitting of $H$ over free abelian groups
of rank at most $n$. By \cite[Prop. 2.2]{Sy} we may assume that
$(\mathcal{H},Y)$ is minimal.

By Theorem \ref{ThmGru}, we have
 $$\begin{array}{lll}
 \sum_{v\in V(Y)}r(H_{v}) &\leq & (2k+1)r(H)+\sum_{e\in
                              E_{+}(Y)}r(H_{e})-2|E_{+}(Y)|-r(Y)\\
 &\leq  & (2k+1)r(H)+n|E_{+}(Y)|-2|E_{+}(Y)|-r(Y)\\
 &\leq  & (2k+1)r(H)+(n-2)(r(Y)+|V(Y)|-1)+r(Y)\\
 &\leq & (2k+1)r(H)+(n-1)r(Y)+(n-2)(|V(Y)|-1)\\
 &\leq & (2k+n)r(H)+(n-2)\big(2k(r(H)-1)\big)
\end{array}$$
where the last inequality follows from Theorem \ref{AcylAccess} (and
the well-known fact that the rank of $H$ bounds the rank of $Y$). It
follows that there are constants $A_{1}$ and $A_{2}$, depending only
on $k$ and $n$, such that $\sum_{v\in V(Y)}r(H_{v})\leq A_{1}\cdot
r(H)+A_{2}$. Hence
$$\begin{array}{lll}
C_{max}^{\varphi}(H)&\leq & r(Y)+\sum_{v\in
V(Y)}\big(1+\varphi(H_{v})\big)\\
&\leq & r(Y)+\sum_{v\in V(Y)}\big(A_{3}\cdot
r(H_{v})+A_{4}+1\big)\\
&\leq & r(Y)+A_{3}\sum_{v\in V(Y)}r(H_{v})+(A_{4}+1)|V(Y)|\\
&\leq & r(H)+A_{3}\big(A_{1}\cdot r(H)+A_{2}\big)+(A_{4}+1)(2k(r(H)-1)+1)\\
\end{array}$$

The above inequality implies that there are constants $A$ and $B$,
depending only on $k$ and $n$, such that
$$C_{max}^{\varphi}(H)\leq A\cdot r(H)+B\leq A\cdot \big([G:H](r(G)-1)+1\big)+B
$$
and therefore
$$\frac{C_{max}^{\varphi}(H)}{[G:H]}
 \leq A\cdot \big(r(G)-1\big)+\frac{A+B}{[G:H]}
$$
which proves the proposition.
\end{proof}

\begin{rem}\label{VacylH1}
In particular, the proof shows that there are constants $C$ and $D$
depending only on $k,n$ such that $V^{\varphi}_{n}(G)\leq C\cdot
r(G)+D$.
\end{rem}

Our first main result is the following.

\firstFinVolOne
\begin{proof}
Since $\varphi(G_{v})>0$ and each edge group is isomorphic to a
proper subgroup of finite index, it follows that $v$ is a
non-degenerate vertex. Let $X$ be the corresponding tree and $v_{0}$
a lift of $v$. Then
\[V(X)=\coprod_{u\in V(Y)}\; G/G_{u}\quad \textrm{and}\quad E_{+}(X)=\coprod_{e\in E_{+}(Y)}\; G/G_{e}.\]
By replacing $X$ with the minimal $G$-invariant
subtree $X_{G}$ of $X$, we may assume that the action is minimal.
Note that $v\in X_{G}/G$ by \cite[Prop. 2.2]{Sy}. We also note that
the collapse of an edge does not increase the ``degree" of
acylindricity. Thus, by collapsing any edge which is not a loop
having degenerate endpoint, we can suppose further that
$(\mathcal{G},Y)$ is reduced.

Let $H$ be a normal subgroup of $G$ of finite index. If the action
of $G$ on $X$ is $k$-acylindrical, then the action of $H$ on $X$ is
$k$-acylindrical as well and each edge $H$-stabilizer is free
abelian of rank at most $n$. Since $H$ is of finite index in $G$, it
follows that $H_{u}$ is of finite index in $G_{u}$ as well, for each
vertex $u$ of $X$. We note that the vertices and edges of the
quotient graph $X/H$ have the following form:
\[V(X/H)=\coprod_{u\in V(Y)}\; H\backslash G/G_{u}\quad \textrm{and}\quad E_{+}(X/H)=\coprod_{e\in E_{+}(Y)}\;
H\backslash G/G_{e}.\]

Let $D_{H}$ be the set of $H$-degenerate vertices of $X$. Since $H$
is normal in $G$, the set $D_{H}$ is invariant under the action of
$G$. If $D_{H}/G=\{v_{1},\dots,v_{m}\}$, then all vertices of type
$H\backslash G/G_{v_{i}}$ for $i=1,\dots,m$ of $X/H$ are degenerate.
For each $i\in\{1,\dots,m\}$, we choose an edge $\psi(v_{i})\in
E_{+}(Y)$ such that $H_{i(\psi(v_{i}))}=H_{v_{i}}$ or
$H_{t(\psi(v_{i}))}=H_{v_{i}}$, and $H_{v_{i}}=G_{v_{i}}\cap
H=G_{\psi(v_{i})}\cap H=H_{\psi(v_{i})}$. In this way we get a map
$\psi:\{v_{1},\dots,v_{m}\}\rightarrow E_{+}(Y)$ such that the
inverse image of $\psi(v_{i})$ under $\psi$ consists of at most two
vertices. Now, if $\textrm{Im}(\psi)=\{e_{1},\dots,e_{k}\}$, then
$\{v_{1},\dots,v_{m}\}=\psi^{-1}(e_{1})\sqcup \cdots
\sqcup\psi^{-1}(e_{k})$.

It is easy to check that the collapse of a degenerate edge not a
loop (i.e. the operation which makes a graph of groups reduced) can
go up complexity by one. Therefore, if we denote the sum $\sum
\varphi(H_{u})$ over all non-degenerate vertices of $X/H$ by $\Phi$,
then
 \begin{align*}
C^{\varphi}_{max}(H) & \geq  C^{\varphi}(\textrm{of a reduced obtained from}\,\, X/H )\\
           & \geq  C^{\varphi}(X/H) =  r(X/H)+|V_{ndeg}(X/H)|+\Phi \\
           & =  |E_{+}(X/H)|-|V(X/H)|+1+|V_{ndeg}(X/H)|+\Phi \\
           & =   \sum_{e\in E_{+}(Y)} |H\backslash G/G_{e}|-
           \sum_{i=1}^{m} |H\backslash G/G_{v_{i}}|+1+\Phi\\
            & =   \sum_{e\in E_{+}(Y)}\frac{[G:H]}{[G_{e}:G_{e}\cap H]}-
           \sum_{i=1}^{m}\frac{[G:H]}{[G_{v_{i}}:G_{v_{i}}\cap H]}
           +1+\Phi, \end{align*}

 where the last equality follows from the normality of $H$ in
 $G$. Hence \small
\begin{align*} \frac{C^{\varphi}_{max}(H)}{[G:H]} & \geq \sum_{e\in
                                  E_{+}(Y)}\frac{1}{[G_{e}:G_{e}\cap H]}-
                                  \sum_{i=1}^{m}\frac{1}{[G_{v_{i}}:G_{v_{i}}\cap H]}
                                  +\frac{1+\Phi}{[G:H]}\\
 & =   \sum_{e\notin \textrm{Im}(\psi)}\frac{1}{[G_{e}:G_{e}\cap H]}+\sum_{e\in
                  \textrm{Im}(\psi)}\frac{1}{[G_{e}:G_{e}\cap H]} -
                  \sum_{i=1}^{m}\frac{1}{[G_{v_{i}}:G_{v_{i}}\cap H]}+\frac{1+\Phi}{[G:H]}\\
 & =    \sum_{e\notin
                  \textrm{Im}(\psi)}\frac{1}{[G_{e}:G_{e}\cap H]}+\sum_{e\in
                  \textrm{Im}(\psi)}\bigg\{\frac{1}{[G_{e}:G_{e}\cap H]}-
                  \sum_{\psi^{-1}(e)}\frac{1}{[G_{v_{i}}:G_{v_{i}}\cap H]}\bigg\}+\frac{1+\Phi}{[G:H]} \\
 & =    \sum_{e\notin \textrm{Im}(\psi)}\frac{1}{[G_{e}:G_{e}\cap H]}+
                  \sum_{e\in \textrm{Im}(\psi)}\bigg\{\frac{1}{[G_{e}:G_{e}\cap H]}-
                  \sum_{\psi^{-1}(e)}\frac{1}{[G_{v_{i}}:G_{e}\cap
                  H]}\bigg\}+\frac{1+\Phi}{[G:H]}.
\end{align*}
\normalsize

 Let us denote the term  $\scalebox{0.9}{$\displaystyle\frac{1}{[G_{e}:G_{e}\cap H]}-
\sum_{\psi^{-1}(e)}\frac{1}{[G_{v_{i}}:G_{e}\cap H]}$}$   by
$A_{e}$. We claim that $A_{e}\geq 0$. Indeed, if $e$ is a loop, then
$\psi^{-1}(e)$ consists of a single vertex, say $v_{i}$, and
      $A_{e}=\scalebox{1.2}{$\frac{1}{[G_{e}:G_{e}\cap
      H]}-\frac{1}{[G_{v_{i}}:G_{e}\cap H]}$}\geq 0$.
 Moreover, $A_{e}=0$
if and only if the vertex of loop $e$ is $G$-degenerate. If
$\psi^{-1}(e)=\{v_{1},v_{2}\}$, then $e$ is not a loop, in
particular $G_{e}$ is properly contained in each of $G_{v_{1}}$ and
$G_{v_{2}}$ (being $Y$ reduced), and thus $2[G_{e}:G_{e}\cap H]\leq
[G_{v_{i}}:G_{e}\cap H]$, $i=1,2$. It follows that
  $\scalebox{0.9}{$ \displaystyle
  A_{e}=\frac{1}{[G_{e}:G_{e}\cap H]}-\frac{1}{[G_{v_{1}}:G_{e}\cap
  H]}-\frac{1}{[G_{v_{2}}:G_{e}\cap H]} \geq \frac{1}{[G_{e}:G_{e}\cap
  H]}\Big(1-\frac{1}{2}-\frac{1}{2}\Big)$}\geq 0$.
 Again, $A_{e}=0$ if
and only if $G_{e}$ is of index two in each of $G_{v_{1}}$ and
$G_{v_{2}}$.

Since the index $[G_{v}:G_{v}\cap H]$ is finite, it follows that
$\varphi(G_{v}\cap H)=[G_{v}:G_{v}\cap H]\varphi(G_{v})>0$. This (as
before) implies that all vertices $gv_{0}$, $g\in G$, are
non-degenerate under the action of $H$ and therefore the orbit of
$v_{0}$ gives $|H\backslash G/G_{v}|=[G:H]/[G_{v}:G_{v}\cap H]$
non-degenerate vertex groups (each isomorphic to $G_{v}\cap H$) in
the quotient graph $X/H$. We conclude that
 \begin{align*}\frac{C_{max}^{\varphi}(H)}{[G:H]}
 &\geq \sum_{e\in E_{+}(Y)-\textrm{Im}(\psi)}\frac{1}{[G_{e}:G_{e}\cap H]}+
                  \sum_{e\in \textrm{Im}(\psi)}A_{e}+\frac{1}{[G:H]}+\frac{|H\backslash
                  G/G_{v}|}{[G:H]}\varphi(G_{v}\cap H)\\
 &\geq \frac{1}{[G_{v}:G_{v}\cap
H]}\varphi(G_{v}\cap H)=\frac{1}{[G_{v}:G_{v}\cap
H]}[G_{v}:G_{v}\cap H]\varphi(G_{v})=\varphi(G_{v}).
 \end{align*}
Finally, $V^{\varphi}_{n}(G)\geq \varphi(G_{v})>0$.
\end{proof}

\begin{rem}\label{fpsmall}
Let $\mathcal{S}$ be a class (closed under subgroups) consisting of
small groups $G$ such that the set of indices of all finite index
subgroups of $G$ is bounded above. It is not difficult to see that
there are infinite groups satisfying this property by considering,
for example, direct products $A\times F$, where $A$ is a (finitely
generated) periodic simple group (such as a Tarski monster) and $F$
is a finite group. Let us denote by $V_{bi}$ the volume defined
using splittings of finitely presented groups over $\mathcal{S}$.
The same arguments as above (with $\varphi=0$) can be used to show
the following:
\begin{thmno} Let $G$ be a finitely presented group which admits a nontrivial
splitting over $\mathcal{S}$ such that at least one vertex group has
no subgroup of finite index contained in $\mathcal{S}$. Then
$V_{bi}(G)>0$.\end{thmno} 

The hypothesis that $G$ is finitely presented is only needed to
insure that $V_{bi}(G)<\infty$. In particular, when $\mathcal{S}$ is
the class of finite groups this gives another proof of \cite[Theorem
4.7]{Sy}: Let $G$ be a finitely presented group (or more generally
an accessible group) with infinitely many ends. Then
$V_{fin}(G)>0$.\end{rem}

We close this section by proving two results concerning a form of
strong Hopficity for finitely generated groups with positive volume.
Recall that a group $G$ is said to be Hopfian if every surjective
endomorphism of $G$ is necessarily an automorphism. A well-known
result of Malcev states that every finitely generated residually
finite group is Hopfian. A deep result of Sela \cite{Se-1} states
that  any  torsion-free hyperbolic group is Hopfian. Note that both
classes of groups (i.e. residually finite and hyperbolic) are closed
under subgroups of finite index.

The following proposition can be thought of as an algebraic analogue
of Theorem 2 in \cite{Rez} (for the case of infinitely many ends see
Proposition \ref{StabImInf}).

\firstImagStab
\begin{proof}
Since each subgroup $f^{k}(G)$ is of finite index in $G$, we have
$[G:f^{k}(G)]=\scalebox{1.2}{$\frac{V^{\varphi}_{n}(f^{k}(G))}{V^{\varphi}_{n}(G)}$}$.
By Remark \ref{VacylH1} the numerator in this expression is bounded
above by $C\cdot r(f^{k}(G))+D$, and therefore by $C\cdot r(G)+D$.
It follows that the set of indices $[G:f^{k}(G)],\,k\in \mathbb{N}$
is bounded. On the other hand, $G$ (being finitely generated) has
finitely many subgroups of bounded finite index, and the first part
of the result follows. In particular, there is $k$ such that the
restriction $f:f^{k}(G)\rightarrow f^{k}(G)$ is onto. Now, the
Hopficity of $f^{k}(G)$ implies that $f$ is an automorphism on
$f^{k}(G)$. Therefore the intersection $\textrm{ker}(f)\cap
f^{k}(G)$ is trivial and thus $\textrm{ker}(f)$ is finite. Since $G$
is torsion-free, we have that $\textrm{ker}(f)=1$ and $f$ is a
monomorphism.  Finally,
$[G:f(G)]=\scalebox{1.2}{$\frac{V^{\varphi}_{n}(f(G))}{V^{\varphi}_{n}(G)}$}=1$,
which shows that $f$ is onto, hence it is an automorphism.
\end{proof}

\begin{rem}
In the case of $0$-acylindrical splittings (i.e. free products),
each finite normal subgroup is trivial and therefore from the
argument in the above proof follows Theorem 2 of \cite{Hi} which can
be stated as follows: Let $G$ be a finitely generated group with
infinitely many ends (in order to have positive volume) which splits
as a free product and whose each subgroup of finite index is
Hopfian. If $f$ is an endomorphism of $G$ with image of finite
index, then $f$ is an automorphism.
\end{rem}

Our argument also leads to the following generalization of Hirshon's
result:

\firstStabImInf
\begin{proof}
We first note that the malnormality assumption implies that each
finite normal subgroup of $G$ is trivial. If $G$ is virtually
torsion-free, then there is a bound $C$ on the orders of its finite
subgroups. It follows, by Lemma \ref{VfinH1}, that $V_{fin}(H)\leq C
\cdot r(H)+1$ for each subgroup of $G$ of finite index. Furthermore,
since $G$ has infinitely many ends, we have $V_{fin}(G)>0$.

Suppose now that $G$ is not virtually torsion-free. Then in each
splitting of $G$ over finite subgroups, there is at least one
infinite vertex group. Thus, by Remark \ref{fpsmall}, there is a
positive integer $m$ such that $V_{fin,m}(G)>0$ (in this case
$\mathcal{S}=\mathcal{S}_{m}$ and $ V_{bi}(G)=V_{fin,m}(G)$). Again
by Lemma \ref{VfinH1}, we have $V_{fin,m}(H)\leq m \cdot r(H)+1$ for
each subgroup of $G$ of finite index. Thus, in either case, the
argument in the proof of Theorem \ref{ImagStab} applies.
\end{proof}

\section{Cyclic splittings of one-ended hyperbolic
groups}\label{SecHyp}

We apply the results of the previous section to cyclic splittings
(i.e. each edge group is infinite cyclic) of one-ended, torsion-free
hyperbolic groups using weighted complexity with
$\varphi_{0}=V_{fin}$ for finitely presented, torsion-free vertex
groups (see Lemma \ref{VfinH1} and Definition \ref{AcyclVol}). We
write $V_{1}^{hyp}$ to denote the corresponding volume. Since such a
splitting is acylindrical, Theorem \ref{FinVolOne} immediately
implies:

\firsthyppos

\begin{exam}\label{surfacetype} Let $(\mathcal{G},Y)$ be a cyclic splitting of a
one-ended, torsion-free hyperbolic group such that each vertex group
$G_{v_{i}}$ is a non-abelian free group of rank $n_{i}$. Suppose
that $E_{+}(Y)=\{e_{1},\ldots,e_{k}\}$ and that
$V(Y)=\{v_{1},\ldots,v_{m}\}$. We denote the edge group $G_{e_{i}}$
by $H_{i}$ and the vertex group $G_{v_{i}}$ by $G_{i}$. Let $X$ be
the universal tree and $H$ a normal subgroup of $G$ of finite index.
Then each vertex of $X$ is non-degenerate under the action of $H$
and we have \[\begin{array}{ccl}
C^{\varphi}(X/H) & \geq & r(X/H)+ \sum_{v\in V_{ndeg}(X/H)}\big(1+V_{fin}(H_{v})\big)\\
           & = & |E_{+}(X/H)|+1+\sum_{v\in X/H}V_{fin}(H_{v}) \\
           & = & \sum_{i=1}^{k} |H\backslash G/H_{i}|+ \sum_{i=1}^{m} |H\backslash G/G_{i}|\cdot V_{fin}(H\cap G_{i})+1\\
           & = & \sum_{i=1}^{k}\scalebox{1.1}{$\frac{[G:H]}{[G_{i}:H\cap G_{i}]}$}+
              \sum_{i=1}^{m}\scalebox{1.1}{$\frac{[G:H]}{[G_{i}:H\cap G_{i}]}$}V_{fin}(H\cap G_{i})+1\\
           & = & \sum_{i=1}^{k}\scalebox{1.1}{$\frac{[G:H]}{[G_{i}:H\cap
           G_{i}]}$}+\sum_{i=1}^{m}[G:H]\cdot (r(G_{i})-1)+1,
\end{array}\]
where the last equality follows from the formula $V_{fin}(H\cap
G_{i})=[G_{i}:H\cap G_{i}]\cdot V_{fin}(G_{i})$ and the fact that
$V_{fin}(F)=r(F)-1$, if $F$ is a free group. Thus
$$\frac{C^{\varphi}(X/H)}{[G:H]}\geq \sum_{i=1}^{m} (r(G_{i})-1)+\sum_{i=1}^{k}\frac{1}{[H_{i}:H\cap
           H_{i}]}+\frac{1}{[G:H]},$$ and hence $V_{1}^{hyp}(G)\geq \sum_{i=1}^{m} (r(G_{i})-1).$

In particular, if $G$ is the fundamental group of a closed
orientable surface $S_{g}$ of genus $g\geq 2$, then the pair of
pants decomposition of $S_{g}$ gives a cyclic splitting for $G$ with
$2g-2$ vertices and $3g-3$ edges such that each vertex group is free
of rank two. Therefore $V_{1}^{hyp}(G)\geq 2(g-1)=-\chi(S_{g})$,
where $\chi(S_{g})$ denotes the (usual) Euler characteristic of
$S_{g}$.
\end{exam}

A particularly important class of cyclic splittings of one-ended
torsion-free hyperbolic groups are JSJ decompositions. There are two
approaches, almost equivalent, both appropriate for our purpose. The
first is due to Sela \cite{Se0} and the other due to Bowditch
\cite{Bow}. We record some of the properties of Bowditch's JSJ
decomposition which is completely canonical.

Let $G$ be a torsion-free, one-ended hyperbolic group. Then there is
a cyclic splitting of $G$, called the JSJ decomposition of $G$, such
that each vertex group is either a maximal infinite cyclic subgroup
of $G$, a maximal hanging Fuchsian subgroup, or a rigid vertex
group. These three types of vertices are mutually exclusive. A
\textsl{maximal hanging Fuchsian} (MHF) subgroup is a non-cyclic
group which is isomorphic to the fundamental group of a compact
surface with boundary and therefore a non-abelian free group. Edge
groups incident to a (MHF) vertex group are precisely the peripheral
subgroups associated to the boundary components. Rigid vertex groups
do not admit splittings over cyclic subgroups relative to the
incident edge groups.

\begin{rem}
Note that although the JSJ decomposition of the fundamental group of
a closed orientable surface $S_{g}$ of genus $g\geq 2$ is trivial,
by the example above we have that its volume is positive.
\end{rem}

Sela \cite{Se0} (see also \cite{Lev}) used the JSJ decomposition of
a one-ended hyperbolic group $G$ to obtain a description of its
outer automorphism group $\textrm{Out}(G)$, from which one can
deduce the existence of (MHF) vertices in the case where
$\textrm{Out}(G)$ is not virtually  finitely generated free abelian.

Since a (MHF) subgroup is a non-abelian free group, by Proposition
\ref{hyppos} we obtain:

\firstJSJpos

\section{One-ended limit groups}\label{SecLim}

A group $G$ is \textsl{fully residually free} (or
$\omega$-residually free) if, for each finite subset $X$ of $G$,
there exists a homomorphism from $G$ to a free group $F$ which is
injective on $X$. A \textsl{limit group} is a finitely generated
fully residually free group. Fully residually free groups have been
extensively studied by Kharlampovich and Myasnikov \cite{KM1,KM2},
and by Remeslennikov \cite{Re} (under the name $\exists$-free
groups).

The term ``limit group" was introduced by Sela \cite{Sel2} and
reflects the fact that they are obtained from limits of sequences of
homomorphisms of an arbitrary finitely generated group to a free
group.

Examples of limit groups include finitely generated free groups,
finitely generated free abelian groups and fundamental groups of
closed hyperbolic surfaces with Euler characteristic less than -1.

Here we use the characterization of limit groups as finitely
generated subgroups of $\omega$-residually free tower groups
\cite{Sel3,KM2} (see also \cite{AB} for a proof).

\begin{defi}[\cite{Sel2}]
An $\omega$-rft space of height $0$ is the wedge of finitely many
circles, $n$-dimensional tori $\mathbb{T}^{n}$ ($n\geq 2$) and
closed hyperbolic surfaces of Euler characteristic less than -1.

An $\omega$-rft space $X_{h}$ of height $h$ is obtained from an
$\omega$-rft space $X_{h-1}$ of height $h-1$ by attaching a block of
one of the two following types:
\begin{description}
\item[1. Abelian block.] $X_{h}$ is the quotient space $X_{h-1}\sqcup
\mathbb{T}^{n}/\sim$, where a coordinate circle in $\mathbb{T}^{n}$
is identified with a nontrivial loop in $X_{h-1}$ that generates a
maximal abelian subgroup in $\pi_{1}(X_{h-1})$ (in fact infinite
cyclic).
\item[2. Quadratic block.] $X_{h}=X_{h-1}\sqcup
\Sigma/\sim$, where $\Sigma$ is a connected compact hyperbolic
surface of Euler characteristic at most -2 or a punctured torus
whose each boundary component is identified with a homotopically
non-trivial loop in $X_{h-1}$. It is also required that there exists
a retraction $r:X_{h}\rightarrow X_{h-1}$ with
$r_{\ast}(\pi_{1}\Sigma)$ non-abelian.
\end{description}
\end{defi}
Applying the Seifert-van Kampen theorem to the decomposition defined
by the final block, we see that the fundamental group of an
$\omega$-rft space $X_{h}$ of height $h\geq 1$ admits a cyclic
splitting as a 2-vertex graph of groups where one of the vertex
groups is $\pi_{1}(X_{h-1})$ and the other free or free abelian of
finite rank at least 2. Moreover, edge groups are maximal infinite
cyclic subgroups in the second type vertex groups. The malnormality
of peripheral subgroups in $\Sigma$ as well as the malnormality of
maximal abelian subgroups in limit groups imply that the above
splitting is $2$-acylindrical (see \cite[Lemma 1.4]{BH2}).

The \textsl{height} of a limit group $G$ is the minimal number $h$
for which there is an $\omega$-rft space $X_{h}$ of height $h$ such
that $\pi_{1}(X_{h})$ contains a subgroup isomorphic to $G$.

We consider cyclic $2$-acylindrical splittings of one-ended limit
groups using weighted complexity with $\varphi_{0}=V_{fin}$ for
finitely generated vertex groups with infinitely many ends. Let us
denote by $V_{1}^{lim}$ the associated volume.

Let $G$ be a one-ended limit group which is isomorphic to the
fundamental group of an $\omega$-rft space $X_{h}$ such that the
final block added is quadratic. Then the 2-acylindrical splitting of
$G$, defined as above by the final block, contains a free
non-abelian group and thus, by Theorem \ref{FinVolOne}, we have
$V_{1}^{lim}(G)>0$. The following shows that the same is true for
any one-ended finitely generated subgroup of height $h$ of $G$.

\firstlimpos

\begin{proof} In this case it follows as in the proof of \cite[Lemma
1.4]{BH1}, that $H$ admits a two-acylindrical splitting as before
which has a non-abelian free group and hence, by Theorem
\ref{FinVolOne}, $V_{1}^{lim}(H)>0$.
\end{proof}

In \cite{Sel4} (see also \cite{KM3}), Sela proved that a finitely
generated group has the same elementary theory as a non-abelian free
group if and only if it is the fundamental group of an $\omega$-rft
space $X_{h}$, whose construction involves only quadratic blocks.
Such groups are called \textsl{elementarily free}.

\begin{cor}\label{Elem}
If $G$ is a one-ended elementarily free group, then
$V_{1}^{lim}(G)>0$.
\end{cor}

\begin{rem} As before, we consider cyclic $2$-acylindrical splittings of one-ended limit
groups of height $h$. If we could show that the height of a limit
group is equal to the height of any finite index subgroup, then we
could define a volume $V_{1}^{l,h}$ inductively as follows:

If $G$ is a one-ended limit group of height $0$, then $G$ is either
a free abelian group of finite rank and thus $V_{1}^{l,0}(G)=0$, or
a hyperbolic surface group and we define
$V_{1}^{l,0}(G)=V_{1}^{hyp}(G)>0$.

If $G$ is a one-ended limit group of height $h>0$, then we use for
any cyclic $2$-acylindrical splitting $(\mathcal{G,Y})$ weighted
complexity defined by $\varphi(G_{v})=V_{fin}(G_{v})>0$, if $G_{v}$
is a vertex group with infinitely many ends, and
$\varphi(G_{v})=V_{1}^{l,k}(G_{v})$, if $G_{v}$ is a one-ended limit
group of height $k\leq h-1$. We denote by $V_{1}^{l,h}$ the
associated volume.

Note that, by Lemma \ref{VfinH1} and Remark \ref{VacylH1},
hypothesis \ref{H} is satisfied in each step. This means that
$V_{1}^{l,h}(G)<\infty$.

By \cite[Lemma 1.4]{BH1}, each one-ended limit group $G$ of height
$h$ admits a cyclic (acylindrical) splitting such that at least one
vertex group is a non-abelian limit group of height $\leq h-1$. Thus
it would follow (inductively) from Theorem \ref{FinVolOne} that
$V_{1}^{l,h}(G)>0$ for any non-abelian, one-ended limit group.
\end{rem}
It is worth noting that if $H$ is a finite index subgroup of a
one-ended, limit group $G$ of height $1$, then the height of $H$ is
also $1$. Indeed, this follows immediately by a well-known theorem
of Eckmann and M\"{u}ller which states that each virtual surface
group is a surface group. Therefore, the above discussion shows the
following:
\begin{cor}\label{Height1} Let $G$ be a one-ended, non-abelian limit group of
height $1$. Then $V_{1}^{l,1}(G)>0$.
\end{cor}
\section{Acylindrical splittings of CSA groups over abelian
subgroups}\label{SecCSA}

In this section, we generalize the preceding results to the case of
one-ended CSA groups. Recall that a group $G$ is \textsl{CSA} (or
conjugately separated abelian) if the maximal abelian subgroups of
$G$ are malnormal. If $\Gamma$ is a torsion-free group which is
hyperbolic relative to a (finite) family of finitely generated free
abelian subgroups, then each $\Gamma$-limit group (or equivalently,
by Theorem \cite[Theorem 5.2]{Gro2}, each finitely generated and
fully residually $\Gamma$) is CSA (see \cite[Lemma 6.9]{Gro1}). We
refer to \cite{Gro1,Gro2} for the basic definitions and properties
concerning $\Gamma$-limit groups.

In particular, limit groups, torsion-free hyperbolic groups and more
generally finitely generated subgroups of torsion-free groups
hyperbolic relative to a collection of free abelian subgroups, are
CSA.

Let $G$ be a finitely generated group. In \cite{GL1}, Guirardel and
Levitt gave a method for constructing an acylindrical $G$-tree
$T_{c}$, called the \textsl{tree of cylinders}, from any $G$-tree
$T$ with edge stabilizers in a class $\mathcal{A}$ closed under
conjugation. Using trees of cylinders they obtain
$\textrm{Out}(G)$-invariant splittings which in many cases give a
lot of information about the structure of $\textrm{Out}(G)$ (in the
same way as in the case of torsion-free hyperbolic groups). The
vertex stabilizers of $T_{c}$ are divided into two types:
\textsl{rigid} and \textsl{flexible}. A vertex stabilizer $G_{v}$ is
rigid if, whenever $G_{v}$ acts on a tree with edge stabilizers in
$\mathcal{A}$, then it stabilizes a vertex, and flexible otherwise
(see \cite{GL2} for more details and for a description of flexible
vertex stabilizers when $\mathcal{A}$ consists of slender groups).

If we restrict to the case of splittings of one-ended CSA groups
over abelian groups, we have the following result.
\begin{thm}(\cite[Prop. 6.3]{GL1}, \cite[Theorem 9.5]{GL2})
Let $G$ be a finitely generated, one-ended and torsion-free CSA
group, let $T$ be a $G$-tree with finitely generated abelian edge
stabilizers, and let $T_{c}$ be the associated tree of cylinders.
\begin{enumerate}
\item Edge stabilizers of $T_{c}$ are non-trivial and abelian.
\item $T$ and $T_{c}$ have the same non-abelian vertex stabilizers;
non-abelian flexible vertex stabilizers are fundamental groups of
compact surfaces with boundary.
\item $T_{c}$ is $2$-acylindrical.
\end{enumerate}
\end{thm}
Let $G$ be a finitely generated, torsion-free, one-ended CSA group
such that abelian subgroups of $G$ are finitely generated of bounded
rank. For each acylindrical splitting of such a group over abelian
subgroups of bounded rank $n$, we use weighted complexity with
$\varphi_{0}=V_{fin}$ (see Definition \ref{AcyclVol}) and we denote
the associated volume by $V_{n}^{CSA}$.

\firstCSA
\begin{proof}
Let $n=n(G)$ be a positive integer that bounds the ranks of the
abelian subgroups of $G$. In particular, $n$ bounds the ranks of the
abelian subgroups of each subgroup $H$ of $G$. The splitting
$(\mathcal{G},T_{c}/G)$ of $G$ corresponding to the action on
$T_{c}$ is acylindrical, $\varphi_{0}(G_{v})=r(G_{v})-1>0$ ($G_{v}$
being non-abelian free), so Theorem \ref{FinVolOne} applies.
\end{proof}

Let $\Gamma$ be a torsion-free group which is hyperbolic relative to
a collection $\{A_{1},\ldots,A_{k}\}$ of finitely generated free
abelian subgroups. Each non-cyclic abelian subgroup of $\Gamma$ is
contained in a conjugate of some $A_{i}$ (see for example the proof
of \cite[Lemma 2.2]{Gro2}) and therefore abelian subgroups of
$\Gamma$ are finitely generated of bounded rank. It is also worth
noting that each non-abelian, one-ended, strict (i.e. it is not a
f.g. subgroup of $\Gamma$) $\Gamma$-limit group $G$ which is not a
surface group, admits a non-trivial abelian splitting \cite[Prop.
2.11]{Gro2}.

\begin{lem} Let $\Gamma$ be as above. If $G$ is a $\Gamma$-limit
group, then abelian subgroups of $G$ are finitely generated of
bounded rank.
\end{lem}

\begin{proof} Abelian subgroups of $G$ are finitely generated by
\cite[Cor. 5.12]{Gro2}. Following the proof (and the terminology) of
\cite[Prop. 5.11]{Gro2}, there is a sequence of shortening quotients
of $G$:
\[G\rightarrow L_{2}\rightarrow \cdots \rightarrow L_{s}\]
such that $L_{s}$ is a free product of a finitely generated free
group and a finite collection of finitely generated subgroups of
$\Gamma$. Moreover, $L_{i}$ is a free product of a finitely
generated free group and freely indecomposable $\Gamma$-limit groups
such that each free factor of $L_{i}$ is either (i) a finitely
generated free group; (ii) a finitely generated subgroup of
$L_{i+1}$; or (iii) can be represented as the fundamental group of a
finite graph of groups in which the vertex groups are finitely
generated subgroups of $L_{i+1}$ and finitely generated free groups,
and the edge groups are finitely generated abelian subgroups.

If $A$ is a non-cyclic abelian subgroup of $L_{s}$, then $A$ is
contained in a free factor and therefore in $\Gamma$. It follows
that its rank is bounded by the maximal rank of the parabolic
subgroups, say $r(A_{i_{0}})$. If now $A$ is a non-cyclic abelian
subgroup of $L_{i}$, then $A$ is contained in a free factor and thus
is either a subgroup of $L_{i+1}$ or acts non-trivially on a tree
with vertex groups finitely generated subgroups of $L_{i+1}$ and
edge groups finitely generated abelian. In the latter case $A$ acts
on the associated tree and the axis of any hyperbolic element of $A$
is an $A$-invariant line. Thus we have an homomorphism from $A$ to
$\mathbb{Z}$ whose kernel is an abelian subgroup of $L_{i+1}$. It
follows that, in each case, the rank of $A$ is bounded above by the
rank of the kernel plus 1. Finally, the rank of an abelian subgroup
of $G$ is bounded by $r(A_{i_{0}})+s-1$.
\end{proof}
By combining the preceding lemma with Proposition \ref{CSA} we
obtain:

\firstCSAlimit

\begin{cor}
Let $\Gamma$ and $G$ be as in the preceding proposition. Then each
endomorphism of $G$ with image of finite index is an automorphism.
\end{cor}
\begin{proof}
By \cite[Cor. 5.5]{Gro2}, each finitely generated subgroup of a
$\Gamma$-limit group is Hopfian, and Theorem \ref{ImagStab} applies.
\end{proof}

It is proved in \cite[Theorem 3.1]{BGHM} that each endomorphism of
$\Gamma$ whose image is of finite index in $\Gamma$, is a
monomorphism.

\begin{cor} Let $G$ be a one-ended group which is either (i) a hyperbolic group with
$\textrm{Out}(G)$ not virtually finitely generated free abelian;
(ii) an elementarily free group; or (iii) a non-abelian limit group
of height $1$. Then every endomorphism of $G$ with image of finite
index is an automorphism.
\end{cor}
\begin{proof}
It follows from Corollaries \ref{JSJpos}, \ref{Elem} and
\ref{Height1}, that in each case the (corresponding) volume of $G$
is positive.
\end{proof}

\section{Acknowledgements}

The author thanks the referee for the
careful reading of the manuscript and useful suggestions.

\noindent
Department of Mathematics\\
National and Kapodistrian University of Athens\\
Panepistimioupolis, GR-157 84, Athens, Greece\\
{\it e-mail}: msykiot@math.uoa.gr

\end{document}